\documentclass[12pt,reqno]{amsart}
\usepackage{amsmath}
\usepackage[english,activeacute]{babel}
\usepackage{inputenc}
\usepackage{amssymb}
\usepackage{amsthm}
\usepackage{mathtools}
\usepackage{graphics,graphicx,tikz}
\usepackage{array}
\usepackage{a4wide}
\allowdisplaybreaks
\usepackage{color, url}
\usepackage{float}
\usepackage{multicol}
\usepackage[shortlabels]{enumitem}
\usepackage[colorlinks=true, citecolor=red, linkcolor=blue, urlcolor=blue]{hyperref}
\theoremstyle{plain}
\newtheorem{theorem}{Theorem}[section]

\theoremstyle{definition}

\newtheorem{remark}[theorem]{Remark}

\begin{document}
\title{$q$-Analogues of some supercongruences related to generalized Van Hamme-type supercongruences}
	\author{Liton Karmakar}
    \address{Department of Mathematics, National Institute of Technology Silchar, Assam 788010, India}
	\email{litonofficial8638@gmail.com}

	\thanks{\textit{Mathematics Subject Classification.}  33D15 · 11B65 · 05A10.\\
	\textit{Keywords and phrases: $q$-Analogue, Cyclotomic polynomial, $q$-Supercongruences}}
	
\begin{abstract}
Recently, Jana and Kalita (Res. Number Theory $\textbf{8}~ (2022),$ Article $54$) obtained certain supercongruences motivated by some generalized Van Hamme type supercongruences, specifically for an integer $\ell \geq 2$ and an odd prime $p$ with $p \equiv -1 \pmod{\ell},$
    \begin{align*}
    \displaystyle \sum_{n=0}^{\frac{p^v + 1}{\ell}} (-1)^n (2 \ell n - 1) (\ell^2 n^2 - \ell n + 1) \frac{(-\frac{1}{\ell})_n^3}{(1)_n^3} \equiv (-1)^{\frac{p^v + p + 2}{\ell}} p^{3v} \pmod{p^{3v+1}},
\end{align*}
and
\begin{align*}
    \displaystyle \sum_{n=0}^{\frac{p^v + 1}{\ell}}  (2 \ell n - 1) (2 \ell^2 n^2 - 2 \ell n + 1) \frac{(-\frac{1}{\ell})_n^4}{(1)_n^4} \equiv - p^{4v} \pmod{p^{4v+1}}.
\end{align*}
Employing the $q$-telescoping technique, similar to the $q$-WZ method, we here establish some supercongruences involving certain $q$-shifted factorials. As particular cases, we provide $q$-analogues of the above supercongruences.
\end{abstract}

\maketitle

\section{Introduction and Statement of Results}\label{section1}

  In $1997$, Van Hamme \cite{hamme} observed numerous conjectural supercongruences relating certain truncated hypergeometric sums  to the values of the $p$-adic gamma function, including Ramanujan\rq s \cite{ramanujan} identity
 \begin{align*}
 \sum_{n=0}^{\infty} \left(-1\right)^n \left(4n+1\right) \frac{\left(\frac{1}{2}\right)_{n}^3}{\left(1\right)_{n}^3} = \frac{2}{\pi}  
 \end{align*}
 admits a nice $p$-adic analogue
 \begin{equation} \label{padic B2}
 \sum_{n=0}^{\frac{p-1}{2}} \left(-1\right)^n \left(4n+1\right) \frac{\left(\frac{1}{2}\right)_{n}^3}{\left(1\right)_{n}^3} \equiv \left(-1\right)^{\frac{p-1}{2}}p \pmod{p^3}.
 \end{equation} 
Using a $_6F_5$ hypergeometric series identity, Mortenson \cite{Mortenson2008} was the first to prove the supercongruence \eqref{padic B2}. Subsequently, Zudilin \cite{Zudilin2009} reconfirmed this result using the WZ-method \cite{PWZ}, and  Long \cite{Long2011} provided further validation via certain hypergeometric series identities and evaluations. In \cite{Guo2018}, Guo  provided the following generalization of \eqref{padic B2} by employing the $q$-WZ method :
\begin{equation} \label{P5 Eqn 1}
\sum_{n=0}^{\frac{p^v-1}{2}} \left(-1\right)^n \left(4n+1\right) \frac{\left(\frac{1}{2}\right)_{n}^3}{\left(1\right)_{n}^3} \equiv p^v \left(-1\right)^{\frac{\left(p-1\right)v}{2}} \pmod{p^{v+2}}.
\end{equation}
Additionally, in \cite{Guo2017}, Guo effectively presented another generalization of \eqref{padic B2} by utilizing the powerful WZ method:
\begin{equation}\label{P5 Eqn 2}
\sum_{n=0}^{\frac{p-1}{2}}
(-1)^n(4n+1)^3
\frac{\left(\frac{1}{2}\right)_n^3}{(1)_n^3}
\equiv
-3p(-1)^{\frac{p-1}{2}}
\pmod{p^2}.
\end{equation}

Subsequently, Jana and Kalita \cite{JanaKalita2019} expanded upon \eqref{P5 Eqn 2}, presenting a more generalized one as follows:
\begin{equation} \label{JK B2 eq 1}
\sum_{n=0}^{\frac{p^v-1}{2}} \left(-1\right)^n \left(4n+1\right)^3 \frac{\left(\frac{1}{2}\right)_{n}^3}{\left(1\right)_{n}^3} \equiv -3p^v \left(-1\right)^{\frac{\left(p-1\right)v}{2}} \pmod{p^{v+2}}.
\end{equation}

In recent years, many scholars have proposed further generalizations of the supercongruence results related to \eqref{P5 Eqn 1} and \eqref{JK B2 eq 1}. For instance, interested readers can refer to \cite{Guo2025B2C2, JanaKalita2019, JanaKalita2021, JanaKarmakar2025}. In line with this, Jana and Kalita \cite{JanaKalita2021} confirmed the following supercongruence conjecture proposed by Guo \cite[Conjecture $4.2$]{Guo2017}, employing a simple telescoping method:
\begin{align} \label{Guo Conjecture B2}
\displaystyle \sum_{n=0}^{\frac{p^{v}-1}{2}} (-1)^n(4n+1)(4n^2 + 2n + 1)\frac{\left(\frac{1}{2}\right)_n^3}{(1)_n^3} \equiv p^{3v}\pmod{p^{3v+1}}.
\end{align}
It is also important to note that there is a similar family of supercongruences associated with \eqref{padic B2} as
\begin{equation} \label{P5 Eqn 3}
\sum_{n=0}^{\frac{p+1}{2}} \left(-1\right)^n \left(4n-1\right) \frac{\left(-\frac{1}{2}\right)_n^3}{\left(1\right)_n^3} \equiv \left(-1\right)^{\frac{p+1}{2}} p \pmod{p^3}.
\end{equation}
In $2022$, motivated by the conjectural supercongruence \eqref{Guo Conjecture B2}, Jana and Kalita \cite{JanaKalita2022} established similar type of supercongruence related to \eqref{P5 Eqn 3} as follows: for integers $v \geq 1$ and $\ell \geq 2,$ and an odd prime $p$ with $p \equiv -1 \pmod{\ell},$
    \begin{align} \label{P5 Task 1}
    \displaystyle \sum_{n=0}^{\frac{p^v + 1}{\ell}} (-1)^n (2 \ell n - 1) (\ell^2 n^2 - \ell n + 1) \frac{(-\frac{1}{\ell})_n^3}{(1)_n^3} \equiv (-1)^{\frac{p^v + p + 2}{\ell}} p^{3v} \pmod{p^{3v+1}}.
\end{align}
Throughout we assume $q$ to be fixed $0 < |q|<1.$ For complex number $a$, the \textit{$q$-shifted factorial} is defined by $$(a;q)_0:=1, (a;q)_n:=\prod_{k=0}^{n-1}(1-aq^k) \text{
for~} n\geq 1, \text{~and~}(a;q)_\infty := \prod_{k=0}^{\infty}(1-aq^k).$$
The \textit{$q$-shifted factorial} for negative index is defined as
\begin{center}
    $\displaystyle (a;q)_{-n}:= \prod_{k=1}^{-n} \left( 1-\frac{a}{q^k} \right).$
\end{center}
Moreover, $
[m]=1+q+q^2+\cdots+q^{\,m-1}
=\frac{1-q^m}{1-q}$ represents the $q$-integer, and the $m$-th cyclotomic polynomial $\Phi_m(q)$ is defined by
$$\Phi_m(q):=\prod_{\substack{1\leq d\leq m\\ \gcd(d,m)=1}}
\left(q-\eta^d\right),$$
where $\eta$ denotes a primitive $m$-th root of unity. It is well known that
$\Phi_m(q)$ is a polynomial in $q$ with integer coefficients. Furthermore,
$$\prod_{\substack{d\mid m\\ d>1}}\Phi_d(q)=[m].$$
Motivated by the aforementioned works, we shall derive some supercongruences involving certain $q$-shifted factorials. Our first result in this article is stated in the following theorem.

\begin{theorem} \label{P5-Thm 1}
     Let $r, \ell \geq 2$ be integers, and $m$ a positive integer with $m \equiv -r \pmod{\ell}$ and \textup{gcd}$(m,\ell)=1$. Let $s$ be a non-negative integer such that $m > \frac{2 \ell s + r}{\ell -1}$. Then, modulo $[m] \Phi_m(q)^2,$ 
    \begin{align} \label{P5-Thm 1 Eqn}
       & \sum_{n=s}^{\frac{m+r}{\ell} + s} (-1)^n \dfrac{(q^{-r};q^{\ell})_{n+s} (q^{-r};q^{\ell})_{n-s} (q^{-r};q^{\ell})_n}{(q^{\ell};q^{\ell})_{n+s} (q^{\ell};q^{\ell})_{n-s} (q^{\ell};q^{\ell})_n} \notag \\
       & ~~\times \left\{ [\ell (n+s) - r] [\ell(n-s)-r] [\ell n -r] + [\ell (n+s)] [\ell (n-s)] [\ell n] \right\} \equiv 0 .
    \end{align}
\end{theorem}
Clearly, the $m=p^v$ and $q \to 1$ case of \eqref{P5-Thm 1 Eqn} leads to the following conclusion:
\begin{align} \label{P5-Thm 1 Eqn 2}
		&\displaystyle \sum_{n=s}^{\frac{p^v + r}{\ell}+s} \left(-1\right)^{n} \left(2\ell n -r\right) \left(\ell^2n^2 - r \ell n -\ell^2 s^2 + r^2 \right) \frac{\left(-\frac{r}{\ell}\right)_{n+s} \left(-\frac{r}{\ell}\right)_{n-s} \left(-\frac{r}{\ell}\right)_{n}}{\left(1\right)_{n+s} \left(1\right)_{n-s} \left(1\right)_{n}}
		\equiv 0 \pmod{p^{v+2}}. 
	\end{align}
\begin{remark}
    Setting $r=1, s=0$ in the above congruence \eqref{P5-Thm 1 Eqn 2}, we obtain congruence \eqref{P5 Task 1} modulo $p^{v+2}$.
\end{remark}

In $2019$, Guo and Schlosser \cite{GuoSchlosser2019 JD}  obtained a notable supercongruence, in relation to \cite[(C.2)]{hamme}, by providing a $q$-analog:
\begin{equation} \label{P5 Eqn 4}
\sum_{n=0}^{\frac{p+1}{2}}  \left(4n-1\right) \frac{\left(-\frac{1}{2}\right)_n^4}{\left(1\right)_n^4} \equiv -5p^4 \pmod{p^5}.
\end{equation}
In $2020$, Guo and Liu \cite[Conjecture $5.2$]{GuoLiu2020} proposed the following general Van Hamme type supercongruence related to \eqref{P5 Eqn 4}: for any odd prime $p$ and positive odd integer $u,$ there exists an integer $a_u$ such that, for any positive integer $v,$ there hold  

\begin{equation} \label{Guo Conjecture B2 2}
    \displaystyle \sum_{n=0}^{\frac{p^{v}+1}{2}} (4n-1)^u \frac{\left(-\frac{1}{2}\right)_n^4}{(1)_n^4} \equiv a_u p^{v}\pmod{p^{v+3}}.
\end{equation}
In particular, we have $a_1= a_3=0, a_5=16, a_7=80, a_9=192, a_{11}=640, a_{13}= - 3472$ and $a_{15}=138480.$ The above supercongruences \eqref{Guo Conjecture B2 2} have been verified by Jana and Kalita \cite{JanaKalita2020} for the cases $u=1$ and $u=3.$ Additionally in \cite{JanaKalita2022}, they established a much stronger supercongruence by observing: $(4n-1)^3 + (4n-1) = 2(4n-1)(8n^2 - 4n +1)$,
\begin{align} \label{P5 Task 2}
    \displaystyle \sum_{n=0}^{\frac{p^v + 1}{\ell}}  (2 \ell n - 1) (2 \ell^2 n^2 - 2 \ell n + 1) \frac{(-\frac{1}{\ell})_n^4}{(1)_n^4} \equiv - p^{4v} \pmod{p^{4v+1}}.
\end{align}

Our second result is presented in the following theorem.
\begin{theorem} \label{P5-Thm 2}
     Let $r, \ell \geq 2$ be integers, and $m$ a positive integer with $m \equiv -r \pmod{\ell}$ and \textup{gcd}$(m,\ell)=1$. Let $s$ be a non-negative integer such that $m > \frac{2 \ell s + r}{\ell -1}$. Then, modulo $[m] \Phi_m(q)^3,$ 
    \begin{align} \label{P5-Thm 2 Eqn}
       & \sum_{n=s}^{\frac{m+r}{\ell} + s} \dfrac{(q^{-r};q^{\ell})_{n+s} (q^{-r};q^{\ell})_{n-s} (q^{-r};q^{\ell})_n^2}{(q^{\ell};q^{\ell})_{n+s} (q^{\ell};q^{\ell})_{n-s} (q^{\ell};q^{\ell})_n^2} \notag \\
       & ~~\times \left\{ [\ell (n+s) - r] [\ell(n-s)-r] [\ell n -r]^2 - [\ell (n+s)] [\ell (n-s)] [\ell n]^2 \right\} \equiv 0 .
    \end{align}
\end{theorem}

Clearly, the $m=p^v$ and $q \to 1$ case of \eqref{P5-Thm 2 Eqn} implies the following:
\begin{align} \label{P5-Thm 2 Eqn 2}
		&\displaystyle \sum_{n=s}^{\frac{p^v + r}{\ell}+s}  \left(2\ell n -r\right) \left(2r\ell^2n^2 - 2 r^2 \ell n -r\ell^2 s^2 + r^3 \right) \frac{\left(-\frac{r}{\ell}\right)_{n+s} \left(-\frac{r}{\ell}\right)_{n-s} \left(-\frac{r}{\ell}\right)_{n}^2}{\left(1\right)_{n+s} \left(1\right)_{n-s} \left(1\right)_{n}^2}
		\equiv 0 \pmod{p^{v+3}}. 
	\end{align}
\begin{remark}
    Setting $r=1, s=0$ in the above congruence \eqref{P5-Thm 2 Eqn 2}, we obtain congruence \eqref{P5 Task 2} modulo $p^{v+3}$.
\end{remark}

\section{Proof of Theorems \ref{P5-Thm 1} and \ref{P5-Thm 2}  }\label{section 2}

\begin{proof} [Proof of Theorem \ref{P5-Thm 1}]
    For non-negative integer $n$, we consider the hypergeometric function $\mathcal{L}$ in $n$ as
    \begin{align*}
        \mathcal{L}(n)= (-1)^{n+1}\dfrac{(q^{-r};q^{\ell})_{n+s} (q^{-r};q^{\ell})_{n-s} (q^{-r};q^{\ell})_n}{(1-q)^3 (q^{\ell};q^{\ell})_{n+s-1} (q^{\ell};q^{\ell})_{n-s-1} (q^{\ell};q^{\ell})_{n-1}}.
    \end{align*}
We have
\begin{align*}
   & \mathcal{L}(n+1) - \mathcal{L}(n) \\
   & = (-1)^{n}\dfrac{(q^{-r};q^{\ell})_{n+s+1} (q^{-r};q^{\ell})_{n-s+1} (q^{-r};q^{\ell})_{n+1}}{(1-q)^3 (q^{\ell};q^{\ell})_{n+s} (q^{\ell};q^{\ell})_{n-s} (q^{\ell};q^{\ell})_{n}} \\ 
   & ~~~~+ (-1)^{n}\dfrac{(q^{-r};q^{\ell})_{n+s} (q^{-r};q^{\ell})_{n-s} (q^{-r};q^{\ell})_n}{(1-q)^3 (q^{\ell};q^{\ell})_{n+s-1} (q^{\ell};q^{\ell})_{n-s-1} (q^{\ell};q^{\ell})_{n-1}} \\
   & = (-1)^{n}\dfrac{(q^{-r};q^{\ell})_{n+s} (q^{-r};q^{\ell})_{n-s} (q^{-r};q^{\ell})_n}{(1-q)^3 (q^{\ell};q^{\ell})_{n+s} (q^{\ell};q^{\ell})_{n-s} (q^{\ell};q^{\ell})_{n}} \\
   & ~~~ \times \left\{ \dfrac{(q^{-r};q^{\ell})_{n+s+1} (q^{-r};q^{\ell})_{n-s+1} (q^{-r};q^{\ell})_{n+1}}{(q^{-r};q^{\ell})_{n+s} (q^{-r};q^{\ell})_{n-s} (q^{-r};q^{\ell})_{n}}  + \dfrac{(q^{\ell};q^{\ell})_{n+s} (q^{\ell};q^{\ell})_{n-s} (q^{\ell};q^{\ell})_n}{ (q^{\ell};q^{\ell})_{n+s-1} (q^{\ell};q^{\ell})_{n-s-1} (q^{\ell};q^{\ell})_{n-1}} \right\} \\
   & = (-1)^{n}\dfrac{(q^{-r};q^{\ell})_{n+s} (q^{-r};q^{\ell})_{n-s} (q^{-r};q^{\ell})_n}{(q^{\ell};q^{\ell})_{n+s} (q^{\ell};q^{\ell})_{n-s} (q^{\ell};q^{\ell})_{n}} \\
   & ~~~\times \left\{ [\ell (n+s) - r] [\ell(n-s)-r] [\ell n -r] + [\ell (n+s)] [\ell (n-s)] [\ell n] \right\}.
\end{align*}
Taking sum on both sides of the above equation with $n$ ranging from $s$ to $\mathcal{M},$ we get
\begin{align} \label{P5-sum 1}
    & \displaystyle \sum_{n=s}^{\mathcal{M}} (-1)^{n}\dfrac{(q^{-r};q^{\ell})_{n+s} (q^{-r};q^{\ell})_{n-s} (q^{-r};q^{\ell})_n}{(q^{\ell};q^{\ell})_{n+s} (q^{\ell};q^{\ell})_{n-s} (q^{\ell};q^{\ell})_{n}} \notag \\
    & ~~~\times \left\{ [\ell (n+s) - r] [\ell(n-s)-r] [\ell n -r] + [\ell (n+s)] [\ell (n-s)] [\ell n] \right\} \notag  \\
    & = \mathcal{L}(\mathcal{M}+1) \notag \\
    & = (-1)^{\mathcal{M}}\dfrac{(q^{-r};q^{\ell})_{\mathcal{M}+s+1} (q^{-r};q^{\ell})_{\mathcal{M}-s+1} (q^{-r};q^{\ell})_{\mathcal{M}+1}}{(1-q)^3 (q^{\ell};q^{\ell})_{\mathcal{M}+s} (q^{\ell};q^{\ell})_{\mathcal{M}-s} (q^{\ell};q^{\ell})_{\mathcal{M}}} \notag \\
    & = (-1)^{\mathcal{M}}\dfrac{ [-r]^3 (q^{\ell-r};q^{\ell})_{\mathcal{M}+s} (q^{\ell-r};q^{\ell})_{\mathcal{M}-s} (q^{\ell-r};q^{\ell})_{\mathcal{M}}}{ (q^{\ell};q^{\ell})_{\mathcal{M}+s} (q^{\ell};q^{\ell})_{\mathcal{M}-s} (q^{\ell};q^{\ell})_{\mathcal{M}}},
\end{align}
where we have used the fact that $\mathcal{L}(s)=0$ as $\frac{1}{(q^{\ell}; q^{\ell})_{-1}}=0$. Now putting $\mathcal{M}= \frac{m+r}{\ell} + s$ in the right-hand side of \eqref{P5-sum 1}, we obtain
\begin{align*}
    & (-1)^{\mathcal{M}}\dfrac{ [-r]^3 (q^{\ell-r};q^{\ell})_{\mathcal{M}+s} (q^{\ell-r};q^{\ell})_{\mathcal{M}-s} (q^{\ell-r};q^{\ell})_{\mathcal{M}}}{ (q^{\ell};q^{\ell})_{\mathcal{M}+s} (q^{\ell};q^{\ell})_{\mathcal{M}-s} (q^{\ell};q^{\ell})_{\mathcal{M}}} \\
    &= (-1)^{\frac{m+r}{\ell} + s}\dfrac{ [-r]^3 (q^{\ell-r};q^{\ell})_{\frac{m+r}{\ell} + 2s} (q^{\ell-r};q^{\ell})_{\frac{m+r}{\ell}} (q^{\ell-r};q^{\ell})_{\frac{m+r}{\ell} + s}}{ (q^{\ell};q^{\ell})_{\frac{m+r}{\ell} + 2s} (q^{\ell};q^{\ell})_{\frac{m+r}{\ell}} (q^{\ell};q^{\ell})_{\frac{m+r}{\ell} + s}} \\
    & \equiv 0 \pmod{\Phi_m(q)^3}.
\end{align*}
The last expression is valid because the $q$-shifted factorial $\left(q^{\ell -r}; q^\ell \right)_n$ can be divisible by $\Phi_m(q)$ for $n \geq \frac{m +r}{\ell}$. Additionally, $\left(q^\ell;q^\ell\right)_n$ is divisible by $\Phi_m(q)$ if $n \geq m.$ When $n<m$ and \textup{gcd}$\left(m,\ell\right)=1$, $\left(q^\ell;q^\ell\right)_n$ is always relatively prime to $\Phi_m(q)$. By combining the numerator and denominator of the above expression, the above congruence holds true due to $m > \frac{2 \ell s + r}{\ell -1}.$
\par \quad The remaining part is to show that Theorem \ref{P5-Thm 1} holds modulo $[m]$ as we know that least common multiple of $[m]$ and $ \Phi_m(q)^3$ is $ [m] \Phi_m(q)^2.$ Equivalently, it suffices to prove that
    \begin{align} \label{P5- E1}
       & \sum_{n=0}^{\frac{m+r}{\ell}} (-1)^n \dfrac{(q^{-r};q^{\ell})_{n+2s} (q^{-r};q^{\ell})_{n} (q^{-r};q^{\ell})_{n+s}}{(q^{\ell};q^{\ell})_{n+2s} (q^{\ell};q^{\ell})_{n} (q^{\ell};q^{\ell})_{n+s}} \notag \\
       & ~~\times \left\{ [\ell (n+2s) - r] [\ell n-r] [\ell (n+s) -r] + [\ell (n+2s)] [\ell n] [\ell (n+s)] \right\} \equiv 0 \pmod{[m]}.
    \end{align}
By following the above, one observes that replacing the upper limit of summation of the above congruence \eqref{P5- E1} by $m-1$ gives the following $q$-congruence:
    \begin{align} 
		& \sum_{n=0}^{m-1} (-1)^n \dfrac{(q^{-r};q^{\ell})_{n+2s} (q^{-r};q^{\ell})_{n} (q^{-r};q^{\ell})_{n+s}}{(q^{\ell};q^{\ell})_{n+2s} (q^{\ell};q^{\ell})_{n} (q^{\ell};q^{\ell})_{n+s}} \notag \\
       & ~~\times \left\{ [\ell (n+2s) - r] [\ell n-r] [\ell (n+s) -r] + [\ell (n+2s)] [\ell n] [\ell (n+s)] \right\} \equiv 0 \pmod{\Phi_m(q)}. \label{P5-E2} 
	\end{align}

    \par Let $\eta \neq 1$ be an $m$-th root of unity, not necessarily primitive. This is to say that $\eta$ is a primitive root of unity of degree $d$ such that $d$ divides $m$ and $d>1.$ There exists an integer $h$ with $0 \leq h \leq d-1$ and $\ell h \equiv -r \pmod{ d}.$ Let $\mathcal{S}_{q}(n)$ stands for the following expression:
    \begin{align*}
      \mathcal{S}_{q}(n) = &  (-1)^n \dfrac{(q^{-r};q^{\ell})_{n+2s} (q^{-r};q^{\ell})_{n} (q^{-r};q^{\ell})_{n+s}}{(q^{\ell};q^{\ell})_{n+2s} (q^{\ell};q^{\ell})_{n} (q^{\ell};q^{\ell})_{n+s}} \notag \\
       & ~~\times \left\{ [\ell (n+2s) - r] [\ell n-r] [\ell (n+s) -r] + [\ell (n+2s)] [\ell n] [\ell (n+s)] \right\}.
    \end{align*}
Now substituting $m=d$ in \eqref{P5-E2} and invoking $\dfrac{\left(q^{-r};q^\ell\right)_n}{\left(q^\ell;q^\ell\right)_n} \equiv 0 \pmod{\Phi_d(q)}$ for   $h<n\leq d-1,$ yields the following:
\begin{center}
    $\displaystyle \sum_{n=0}^{d-1} \mathcal{S}_{\eta}(n) = \sum_{n=0}^{h} \mathcal{S}_{\eta} (n)= 0.$
\end{center}
Observing the relation
\begin{center}
    $\displaystyle \dfrac{\mathcal{S}_{\eta} \left(\beta d +n \right)}{\mathcal{S}_{\eta} \left(\beta d\right)}  = \lim_{q \to \eta} \dfrac{\mathcal{S}_{q} \left(\beta d +n \right)}{\mathcal{S}_{q} \left(\beta d\right)} =\dfrac{\mathcal{S}_{\eta} \left(n \right)}{\mathcal{S}_{\eta}(0)},$
\end{center}
we have
\begin{align*}
    \displaystyle  \sum_{n=0}^{\frac{m + r}{\ell}} \mathcal{S}_{\eta}(n) & = \sum_{\beta=0}^{\frac{m - \ell h + r}{\ell d}-1} \sum_{n=0}^{d-1} \mathcal{S}_{\eta}(\beta d + n) + \sum_{n=0}^{h} \mathcal{S}_{\eta} \left( \frac{m - \ell h + r}{\ell } + n \right)\\
    & = \frac{1}{\mathcal{S}_{\eta}(0)}  \sum_{\beta=0}^{\frac{m - \ell h + r}{\ell d}-1} \mathcal{S}_{\eta} (\beta d) \sum_{n=0}^{d-1} \mathcal{S}_{\eta}(n) + \frac{1}{\mathcal{S}_{\eta}(0)} \mathcal{S}_{\eta} \left(\frac{m - \ell h + r}{\ell } \right)  \sum_{n=0}^{h} \mathcal{S}_{\eta} \left( n \right) \\
    & = 0.
\end{align*}
This implies that $\displaystyle \sum_{n=0}^{\frac{m + r}{\ell}} \mathcal{S}_{q}(n)$ is congruent to $0$ modulo $\Phi_d(q).$ Since every cyclotomic polynomial $\Phi_d(q)$ is irreducible in the ring $\mathbb{Z}[q],$ we establish that the left-hand side of \eqref{P5- E1} is congruent to $0$ modulo
\begin{center}
    $\displaystyle  \prod_{d | m , d>1} \Phi_d(q) = \left[m \right].$
\end{center}
This completes the proof.

\end{proof}

\begin{proof}[Proof of Theorem \ref{P5-Thm 2}]

        For non-negative integer $n$, we consider
    \begin{align*}
        \mathcal{L}(n)= \dfrac{(q^{-r};q^{\ell})_{n+s} (q^{-r};q^{\ell})_{n-s} (q^{-r};q^{\ell})_n^2}{(1-q)^4 (q^{\ell};q^{\ell})_{n+s-1} (q^{\ell};q^{\ell})_{n-s-1} (q^{\ell};q^{\ell})_{n-1}^2}.
    \end{align*}
We have
\begin{align*}
   & \mathcal{L}(n+1) - \mathcal{L}(n) \\
   & = \dfrac{(q^{-r};q^{\ell})_{n+s+1} (q^{-r};q^{\ell})_{n-s+1} (q^{-r};q^{\ell})_{n+1}^2}{(1-q)^4 (q^{\ell};q^{\ell})_{n+s} (q^{\ell};q^{\ell})_{n-s} (q^{\ell};q^{\ell})_{n}^2} \\
   & ~~~ - \dfrac{(q^{-r};q^{\ell})_{n+s} (q^{-r};q^{\ell})_{n-s} (q^{-r};q^{\ell})_n^2}{(1-q)^4 (q^{\ell};q^{\ell})_{n+s-1} (q^{\ell};q^{\ell})_{n-s-1} (q^{\ell};q^{\ell})_{n-1}^2} \\
   & = \dfrac{(q^{-r};q^{\ell})_{n+s} (q^{-r};q^{\ell})_{n-s} (q^{-r};q^{\ell})_n^2}{(1-q)^4 (q^{\ell};q^{\ell})_{n+s} (q^{\ell};q^{\ell})_{n-s} (q^{\ell};q^{\ell})_{n}^2} \\
   & ~~~ \times \left\{ \dfrac{(q^{-r};q^{\ell})_{n+s+1} (q^{-r};q^{\ell})_{n-s+1} (q^{-r};q^{\ell})_{n+1}^2}{(q^{-r};q^{\ell})_{n+s} (q^{-r};q^{\ell})_{n-s} (q^{-r};q^{\ell})_{n}^2}  - \dfrac{(q^{\ell};q^{\ell})_{n+s} (q^{\ell};q^{\ell})_{n-s} (q^{\ell};q^{\ell})_n^2}{ (q^{\ell};q^{\ell})_{n+s-1} (q^{\ell};q^{\ell})_{n-s-1} (q^{\ell};q^{\ell})_{n-1}^2} \right\} \\
   & = \dfrac{(q^{-r};q^{\ell})_{n+s} (q^{-r};q^{\ell})_{n-s} (q^{-r};q^{\ell})_n^2}{(q^{\ell};q^{\ell})_{n+s} (q^{\ell};q^{\ell})_{n-s} (q^{\ell};q^{\ell})_{n}^2} \\
   & ~~~ \times \left\{ [\ell (n+s) - r] [\ell(n-s)-r] [\ell n -r]^2 - [\ell (n+s)] [\ell (n-s)] [\ell n]^2 \right\}.
\end{align*}
Taking sum on both sides of the above equation with $n$ ranging from $s$ to $\mathcal{M},$ we get
\begin{align} \label{P5-sum 2}
    & \displaystyle \sum_{n=s}^{\mathcal{M}}  (-1)^{n}\dfrac{(q^{-r};q^{\ell})_{n+s} (q^{-r};q^{\ell})_{n-s} (q^{-r};q^{\ell})_n^2}{(q^{\ell};q^{\ell})_{n+s} (q^{\ell};q^{\ell})_{n-s} (q^{\ell};q^{\ell})_{n}^2} \notag \\
    &~~~ \times \left\{ [\ell (n+s) - r] [\ell(n-s)-r] [\ell n -r]^2 - [\ell (n+s)] [\ell (n-s)] [\ell n]^2 \right\} \notag \\
    & = \mathcal{L}(\mathcal{M}+1) \notag  \\
    & = \dfrac{(q^{-r};q^{\ell})_{\mathcal{M}+s+1} (q^{-r};q^{\ell})_{\mathcal{M}-s+1} (q^{-r};q^{\ell})_{\mathcal{M}+1}^2}{(1-q)^4 (q^{\ell};q^{\ell})_{\mathcal{M}+s} (q^{\ell};q^{\ell})_{\mathcal{M}-s} (q^{\ell};q^{\ell})_{\mathcal{M}}^2} \notag \\
    & = \dfrac{ [-r]^4 (q^{\ell-r};q^{\ell})_{\mathcal{M}+s} (q^{\ell-r};q^{\ell})_{\mathcal{M}-s} (q^{\ell-r};q^{\ell})_{\mathcal{M}}^2}{ (q^{\ell};q^{\ell})_{\mathcal{M}+s} (q^{\ell};q^{\ell})_{\mathcal{M}-s} (q^{\ell};q^{\ell})_{\mathcal{M}}^2}.
\end{align}
Now putting $\mathcal{M}= \frac{m+r}{\ell} + s$ in the right-hand side of \eqref{P5-sum 2}, we obtain
\begin{align*}
    & \dfrac{ [-r]^4 (q^{\ell-r};q^{\ell})_{\mathcal{M}+s} (q^{\ell-r};q^{\ell})_{\mathcal{M}-s} (q^{\ell-r};q^{\ell})_{\mathcal{M}}^2}{ (q^{\ell};q^{\ell})_{\mathcal{M}+s} (q^{\ell};q^{\ell})_{\mathcal{M}-s} (q^{\ell};q^{\ell})_{\mathcal{M}}^2} \\
    &= \dfrac{ [-r]^4 (q^{\ell-r};q^{\ell})_{\frac{m+r}{\ell} + 2s} (q^{\ell-r};q^{\ell})_{\frac{m+r}{\ell}} (q^{\ell-r};q^{\ell})_{\frac{m+r}{\ell} + s}^2}{ (q^{\ell};q^{\ell})_{\frac{m+r}{\ell} + 2s} (q^{\ell};q^{\ell})_{\frac{m+r}{\ell}} (q^{\ell};q^{\ell})_{\frac{m+r}{\ell} + s}^2} \\
    & \equiv 0 \pmod{\Phi_m(q)^4}.
\end{align*}
Following the similar argument as in the proof of Theorem \ref{P5-Thm 1}, we finish the proof.

\end{proof}

\section{Acknowledgement} 
The author is extremely grateful to Dr. Arijit Jana, for his guidance and encouragement.

	
	{\bf Declarations.} The author make the following declaration.
    \begin{enumerate}
        \item The author do not have any conflicts of interest.
        \item The author has no relevant financial or non-financial interests to disclose.
    \end{enumerate}
    

\end{document}